\documentclass[12pt]{article}
\usepackage[margin=1in]{geometry}
\usepackage{amsmath,amsthm,amssymb}
\usepackage{mathrsfs}
\usepackage{amsmath}
\usepackage{amssymb}
\usepackage{epstopdf}
\usepackage{graphicx} 
\usepackage[utf8]{inputenc}
\usepackage{pgf,tikz}
\usetikzlibrary{arrows}
\usepackage{tikz-cd}
\usepackage{amsfonts}
\usepackage{fancyhdr}
\usepackage{lastpage}
\usepackage[font=small,labelfont=bf]{caption}

\newcounter{daggerfootnote}

\usetikzlibrary{arrows}
\tikzset{commutative diagrams/.cd,arrow style=tikz,diagrams={>=latex'}}

\usetikzlibrary{positioning}

\newcommand{\RR}{\mathbb{R}}
\newcommand{\conv}{\textnormal{conv}}

\theoremstyle{plain}
\newtheorem{Th}{Theorem}[section]
\newtheorem{Lem}[Th]{Lemma}
\newtheorem{Prob}[Th]{Problem}
\newtheorem{Cor}[Th]{Corollary}

\newtheorem{Claim}[Th]{Claim}
\theoremstyle{definition}

\newtheorem{Obs}[Th]{Observation}

\theoremstyle{plain}
\theoremstyle{remark}

\numberwithin{equation}{section}
\date{}
\begin{document}

\title{A positive fraction mutually avoiding sets theorem}

\author{Mozhgan Mirzaei\thanks{Department of Mathematics,  University of California at San Diego, La Jolla, CA, 92093 USA.  Supported by NSF grant DMS-1800736. Email:
{\tt momirzae@ucsd.edu}.}\and Andrew Suk\thanks{Department of Mathematics,  University of California at San Diego, La Jolla, CA, 92093 USA. Supported by NSF grant DMS-1800736, an NSF CAREER award, and an Alfred Sloan Fellowship. Email: {\tt asuk@ucsd.edu}.} }

\maketitle
\begin{abstract}

Two sets $A$ and $B$ of points in the plane are \emph{mutually avoiding} if no line generated by any two points in $A$ intersects the convex hull of $B$, and vice versa.  In 1994, Aronov, Erd\H os, Goddard, Kleitman, Klugerman, Pach, and Schulman showed that every set of $n$ points in the plane in general position contains a pair of mutually avoiding sets each of size at least $\sqrt{n/12}$.  As a corollary, their result implies that, for every set of $n$ points in the plane in general position, one can find at least $\sqrt{n/12}$ segments, each joining two of the points, such that these segments are pairwise crossing.

In this note, we prove a fractional version of their theorem:  for every $k > 0$ there is a constant $\varepsilon_k > 0$ such that any sufficiently large point set $P$ in the plane contains $2k$ subsets $A_1,\ldots, A_{k},B_1,\ldots, B_k$, each of size at least $\varepsilon_k|P|$, such that every pair of sets $A = \{a_1,\ldots, a_k\}$ and $B = \{b_1,\ldots, b_k\}$, with $a_i \in A_i$ and $b_i \in B_i$, are mutually avoiding.   Moreover, we show that $\varepsilon_k = \Omega(1/k^4)$.  Similar results are obtained in higher dimensions.
\end{abstract}

\section{Introduction}

Let $P$ be an $n$-element point set in the plane in general position, that is, no three members are collinear.  For $k > 0$, we say that $P$ contains a \emph{crossing family} of size $k$ if there are $k$ segments whose endpoints are in $P$ that are pairwise crossing.  Crossing families were introduced in 1994 by Aronov, Erd\H os, Goddard, Kleitman, Kluggerman, Pach, and Schulman~{\cite{1}}, who showed that for any given set of $n$ points in the plane in general position, there exists a crossing family of size at least $\sqrt{n/12}$. They raised the following problem (see also Chapter~9 in \cite{brass}).

\begin{Prob}[\cite{1}]
Does there exist a constant $c > 0$ such that every set of $n$ points in the plane in general position contains a crossing family of size at least $cn$?
\end{Prob}

\noindent There have been several results on crossing families over the past several decades~{\cite{3,4,S}}. Very recently, Pach, Rudin, and Tardos showed that any set of $n$ points in general position in the plane determines $n^{1-o(1)}$ pairwise crossing segments. More precisely, they proved the following theorem.

\begin{Th}[\cite{PRT}]
Any set $P$ of $n$ points in general position in the plane determines at least $n / 2^{O(\sqrt{\log n )}}$ pairwise crossing segments.
\end{Th}

The result of Aronov et al.~on crossing families was actually obtained by finding point sets that are \emph{mutually avoiding}.   Let $A$ and $B$ be two disjoint point sets in the plane. We say that $A$ \emph{avoids} $B$ if no line subtended by a pair of points in $A$ intersects the convex hull of $B.$  The sets $A$ and $B$ are \emph{mutually avoiding} if $A$ avoids $B$ and $B$ avoids $A.$  In other words, $A$ and $B$ are mutually avoiding if and only if each point in $A$ "sees" every point in $B$ in the same clockwise order, and vice versa.  Hence two mutually avoiding sets $A$ and $B$, where $|A| = |B| = k$, would yield a crossing family of size $k$.  See Figure $1.$

\begin{figure}[t]
\includegraphics[width=8cm]{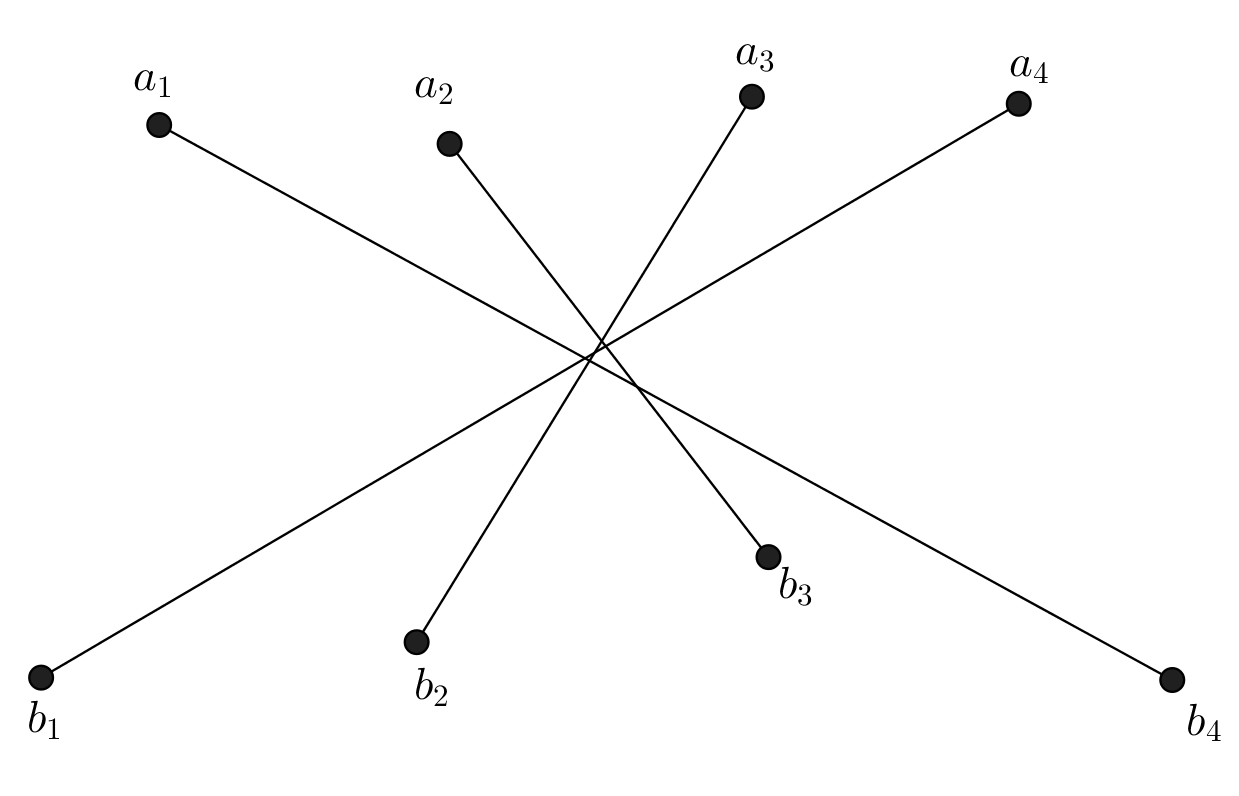}
\centering
\caption{Two mutually avoiding sets $A=\{a_1, a_2, a_3, a_4\}$ and $B=\{b_1, b_2, b_3, b_4\}$ yield a crossing family of size four.}
\end{figure}

\begin{Th}[\cite{1}]\label{crossingr2}
Any set of $n$ points in the plane in general position contains a pair of mutually avoiding sets, each of size at least $\sqrt{n/12}$.

\end{Th}

It was shown by Valtr {\cite{2}} that this bound is best possible up to a constant factor.  In this note, we give a fractional version of Theorem \ref{crossingr2}.

\begin{Th}\label{main}
For every $k > 0$ there is a constant $\varepsilon_k > 0$ such that every sufficiently large point set $P$ in the plane in general position contains $2k$ disjoint subsets $A_1,\ldots, A_{k},B_1,\ldots, B_k$, each of size at least $\varepsilon_k|P|$, such that every pair of sets $A = \{a_1,\ldots, a_k\}$ and $B = \{b_1,\ldots, b_k\}$, with $a_i \in A_i$ and $b_i \in B_i$, are mutually avoiding.   Moreover, $\varepsilon_k = \Omega(1/k^4)$.

\end{Th}

As an immediate corollary, we establish the following fractional version of the crossing families theorem.

\begin{Th}\label{main2}
For every $k > 0$ there is a constant $\varepsilon_k > 0$ such that every sufficiently large point set $P$ in the plane in general position contains $2k$ subsets $A_1,\ldots, A_{k},B_1,\ldots, B_k$, each of size at least $\varepsilon_k|P|$, such that every segment that joins a point from $A_i$ and $B_{k+1-i}$ crosses every segment that joins a point from $A_{k + 1 -i}$ and $B_i$, for $1 \leq i \leq k$.  Moreover, $\varepsilon_k = \Omega(1/k^4)$.
\end{Th}

Let us remark that if we are not interested in optimizing $\varepsilon_k$ in the theorems above, one can combine the well-known same-type lemma due to Barany and Valtr \cite{BV} (see section \ref{same}) with Theorem \ref{crossingr2} to establish Theorems \ref{main} and \ref{main2} with $\varepsilon_k =2^{-O(k^4)}$.  Hence, the main advantage in the theorems above is that $\varepsilon_k$ decays only polynomially in $k$. We will however, use this approach in higher dimensions with a more refined same-type lemma.

\medskip

\noindent \textbf{Higher dimensions.}  Mutually avoiding sets in $\RR^d$ are defined similarly.  A point set $P$ in $\RR^d$ is in \emph{general position} if no $d+1$ members of $P$ lie on a common hyperplane.  Given two point sets $A$ and $B$ in $\RR^d$, we say that $A$ \emph{avoids} $B$ if no hyperplane generated by a $d$-tuple in $A$ intersects the convex hull of $B$.  The sets $A$ and $B$ are \emph{mutually avoiding} if $A$ avoids $B$ and $B$ avoids $A$.  Aronov et al. proved the following.

\begin{Th}[\cite{1}]\label{avoidrd}
For fixed $d \geq 3$, any set of $n$ points in $\RR^d$ in general position contains a pair of mutually avoiding subsets each of size $\Omega_d(n^{1/(d^2 - d + 1)})$.
\end{Th}

\noindent In the other direction, Valtr showed in \cite{2} that by taking a $k\times \cdots \times k$ grid, where $k = \lfloor n^{1/d}\rfloor$, and slightly perturbing the $n$ points so that the resulting set is in general position, one obtains a point set that does not contain mutually avoiding sets of size $cn^{1 - 1/d}$, where $c = c(d)$.

Our next result is a fractional version of Theorem \ref{avoidrd}.

\begin{Th}\label{rd}
For $d \geq 3$ and $k \geq 2$, there is a constant $\varepsilon_{d,k}$, such that every sufficiently large point set $P$ in $\RR^d$ in general position contains $2k$ subsets $A_1,\ldots, A_{k},B_1,\ldots, B_k$, each of size at least $\varepsilon_k |P|$, such that every pair of sets $A = \{a_1,\ldots, a_k\}$ and $B = \{b_1,\ldots, b_k\}$, with $a_i \in A_i$ and $b_i \in B_i$, are mutually avoiding.  Moreover, $\varepsilon_{d,k} = 1/k^{c_d}$ where $c_d > 0$ depends only on $d$.
\end{Th}

\noindent Similar to Theorem \ref{main}, $\varepsilon_{d,k}$ in Theorem \ref{rd} also decays only polynomially in $k$ for fixed $d \geq 3$.  However, $c_d$ does have a rather bad dependency on $d$, $c_d \approx 2^{O(d)}$.

\medskip

 Finally, we establish a result on crossing families in higher dimensions which was also observed by Aronov et al. in \cite{1}.
\section{Proof of Theorem \ref{main}}
\begin{proof} In this section we give the proof of Theorem \ref{main} which closely follows an argument of P\'or and Valtr in \cite{PVPV}. Let $k > 2$ and let $P$ be a set of $n$ points in the plane in general position where $n > (1500k)^4$.  It follows from Theorem $1.2$ that among any $12(40k+1)^2$ points $P$, it is always possible to find two mutually avoiding sets $A \subseteq P$ and $B \subseteq P$ each of size at least $40k+1.$ It follows that $P$ contains at least

\begin{align}\label{2.1}
\frac{\binom{n}{12(40k+1)^2}}{\binom{n-(80k+2)}{12(40k+1)^2-(80k+2)}} = \frac{\binom{n}{80k+2}}{\binom{12(40k+1)^2}{80k+2}}
\end{align}
pairs of mutually avoiding sets, each set of size $40k+1.$ Note that (\ref{2.1}) follows from the equality $$\frac{\binom{m}{a}}{\binom{m-b}{a-b}}=\frac{\binom{m}{b}}{\binom{a}{b}},$$
for positive integers $m,a,b$ where $1 \leq b \leq a \leq m.$\\

Let $A$ and $B$ be a pair of mutually avoiding sets each of size $40k+1$.  For $b \in B,$ label the points in $A$ with $a_1,\ldots, a_{40k+1}$ in radial clockwise order with respect to $b$.  Likewise, for $a \in A$, label the points in $B$ with $b_1,\ldots, b_{40k+1}$ in radial counterclockwise order with respect to $a$.  We say that the pair $(A',B')$ \emph{supports} the pair $(A,B)$ if $A'=\{a_i \in A; i \equiv 1 \textrm{ mod } 4\}$ and $B'= \{b_i \in B; i \equiv 1 \textrm{ mod } 4\}$.   Clearly, $|A'|=|B'|=10k+1.$

Since $P$ has at most $\binom{n}{10k+1}^{2}$ pairs of disjoint subsets with size $10k+1$ each, there is a pair of subsets $(A', B')$ such that $A',B' \subset P, |A'| = |B'| = 10k+1,$ and $(A',B')$ supports at least

\begin{align*}
                 \frac{\binom{n}{80k+2}}{\binom{12(40k+1)^2}{80k+2}  \binom{n}{10k+1}^2} 
                 &> \frac{\left(\frac{n}{80k+2}\right)^{80k+2}}{\left(\frac{12(40k+1)^2e}{80k+2}\right)^{80k+2}\left(\frac{ne}{10k+1}\right)^{20k+2}} \\
                  &> \frac{n^{60k}}{e^{100k+4} 12^{80k+2} (50k)^{141k}} \\ 
                  &> \frac{n^{60k}}{(430k)^{141k}}
\end{align*}

\noindent mutually avoiding pairs $(A,B)$ in $P,$ where $|A| = |B| = 40k+1.$ 
Notice that for the first inequality, we use the inequality $\left(\frac{m}{r}\right)^r < \binom{m}{r} < \left(\frac{me}{r}\right)^r,$ where $1 < r < m.$ To see why the second inequality holds, we claim that $$ \frac{(10k+1)^{20k+2}}{(40k+1)^{160k+4}} > \frac{1}{(50k)^{141k}} \text{ ~as long as~} k > 2.$$ To prove the claim, we need to show that $$(50k)^{141k} > (40k+1)^{140k+2} \left(\frac{40k+1}{10k+1}\right)^{20k+2}.$$ Since $k > 2,$ $(40k+1)^{140k+2} \left(\frac{40k+1}{10k+1}\right)^{20k+2} <(40k+1)^{141k} (\frac{40k+1}{10k+1})^{21k}.$ Therefore, it is enough to show $$(50k)^{141}(10k+1)^{21} >(40k+1)^{162}.$$ It is easy to check that $50^{141}10^{21} > (40.5)^{162}$ (since $k > 2, ~40k+1 < 40.5k$) and this completes the proof of the claim. For the last inequality, it is easy to observe that $e^{100k+4} 12^{80k+2} (50)^{141k} < (430)^{141k},$ for $k>2.$ Note that  
\begin{align*}
e^{100k} < \left(\frac{43}{5}\right)^{46.5k} ~~~\textrm {and }~~~
12^{80k} < \left(\frac{43}{5}\right)^{92.5k}.
\end{align*}
Therefore,
\begin{align*}
e^{100k}12^{80k}12^2 e^4 < \left(\frac{43}{5}\right)^{46.5k} \left(\frac{43}{5}\right)^{92.5k} \left(\frac{43}{5}\right)^{5} < \left(\frac{43}{5}\right)^{141k}.
\end{align*}

\medskip

Set $A'=\{a'_1, \ldots, a'_{10k + 1}\}$ and $B'=\{b'_1, \ldots, b'_{10k + 1}\}.$  For any two consecutive points $a'_i,a'_{i+1} \in A',1 \leq i \leq 10k,$ consider the region $\mathcal{A}_i$ produced by the intersection of regions bounded by the lines $b'_1a'_i, b'_1a'_{i+1}$ and $b'_{10k}a'_i, b'_{10k}a'_{i+1}.$
Similarly, we define the region $\mathcal{B}_i$ produced by the intersection of regions bounded by the lines $a'_1b'_i, a'_1b'_{i+1}$ and $a'_{10k}b'_i, a'_{10k}b'_{i+1}$ for $1 \leq i \leq 10k.$ Therefore, we have $20k$ regions $\mathcal{A}_1, \ldots, \mathcal{A}_{10k}, \mathcal{B}_1, \ldots, \mathcal{B}_{10k}.$ See Figure \ref{2}. 

\begin{figure}[t]
\includegraphics[width=11cm]{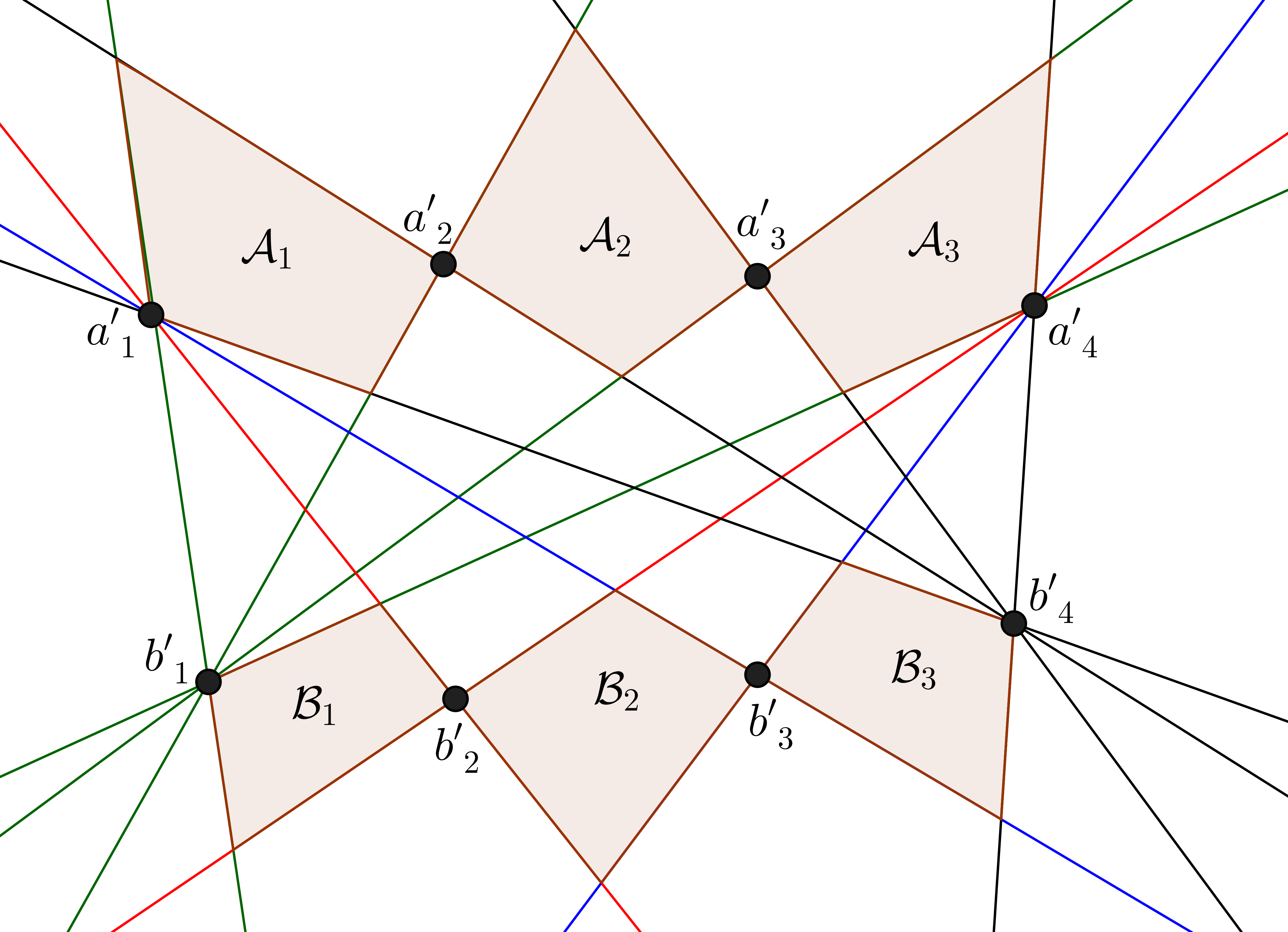}
\centering
\caption{The regions $\mathcal{A}_i$ and $\mathcal{B}_i$ defined by support $A'=\{a'_1, a'_2, a'_3, a'_4\}$ and $B'=\{b'_1, b'_2, b'_3, b'_4\}.$ Let us remark that $4 \neq 10k + 1$ for $k \in \mathbb{Z}.$  The purpose of this figure is to give some intuition on how the regions $A_i$ and $B_i$ are formed.
}\label{2}
\end{figure}

\begin{Obs}
Let $A$ and $B$ be a pair of mutually avoiding sets each of size $40k + 1$. If $(A',B')$ supports $(A,B)$, where $A'= \{a'_1, \ldots, a'_{10k + 1}\}$ and $B'= \{b'_1, \ldots, b'_{10k + 1}\}$, then $A = A' \cup A_1\cup\cdots \cup A_{10k}$ and $B = B' \cup B_1\cup \cdots \cup B_{10k}$, where $|A_i|=|B_i|=3$ for all $ 1 \leq i \leq  10k,$ and $A_i$ lies in region $\mathcal{A}_i$ and $B_i$ lies in region $\mathcal{B}_i$.
\end{Obs}

For $i=1, \ldots, 10k,$ let $\alpha_i,$ respectively $\beta_i,$ denote the number of points of $P$ lying in the interior of $\mathcal{A}_i,$ respectively $\mathcal{B}_i.$ It follows from Observation $2.1$ that $(A',B')$ supports at most $\prod_{i=1}^{10k} \binom{\alpha_i}{3}\prod_{i=1}^{10k}\binom{\beta_i}{3}$ pairs of mutually avoiding sets $(A,B)$, each of size $40k+1$.  Therefore,

\begin{align*}
\frac{n^{60k}}{(430k)^{141k}} \leq  \prod_{i=1}^{10k} \binom{\alpha_i}{3}\prod_{i=1}^{10k}\binom{\beta_i}{3} \leq \prod_{i=1}^{10k} (\alpha_i\beta_i)^3.
\end{align*}

\noindent Without loss of generality, let us relabel the regions $\mathcal{A}_1,\ldots, \mathcal{A}_{10k}, \mathcal{B}_1,\ldots, \mathcal{B}_{10k}$ so that $\alpha_1 \leq \alpha_2 \leq \cdots \leq \alpha_{10k}$ and $\beta_1 \leq \beta_2 \leq \cdots \leq \beta_{10k}.$\\

\begin{Claim}
There exists an $i$ such that $1 \leq i \leq 9k,$ and $\alpha_i,\beta_i \geq \frac{n}{(1320k)^4}.$
\end{Claim}

\begin{proof}
For the sake of contradiction, suppose for each $i,$ $1 \leq i \leq 9k,$ we have $\alpha_i < \frac{n}{(1320k)^4}.$
Therefore,
\begin{align*}
\frac{n^{20k}}{(430k)^{47k}}\leq \prod_{i=1}^{10k} \alpha_i\beta_i
&= \prod_{i=1}^{9k} \alpha_i \left( \prod_{i=9k+1}^{10k}\alpha_i \prod_{i=1}^{10k} \beta_i )\right)\\
& \leq \left(\frac{n}{(1320k)^4}\right)^{9k} \left(\frac{\Sigma_{i=9k+1}^{10k} \alpha_i + \Sigma_{i=1}^{10k} \beta_i}{11k}\right)^{11k}\\
& < \left(\frac{n}{(1320k)^4}\right)^{9k} \left(\frac{n}{11k}\right)^{11k}\\
&=  \frac{n^{20k}}{(1320k)^{{36k}}({11k})^{11k}}.\\
\end{align*}
Hence, we have 
\begin{align}\label{2.2}
\frac{n^{20k}}{(430k)^{47k}} < \frac{n^{20k}}{(1320k)^{{36k}}({11k})^{11k}} \cdot
\end{align}
After simplifying (\ref{2.2}), we get $ \frac{1320^{36}11^{11}}{430^{47}} < 1$ which is a contradiction as $ \frac{1320^{36}11^{11}}{430^{47}} \approx 1.054$. Thus, there exists an $i,$ $1 \leq i \leq 9k$, with $\alpha_i \geq \frac{n}{(1320k)^4}.$ With a similar calculation, there exists an $i,$ $1 \leq i \leq 9k$ with $\beta_i \geq \frac{n}{(1320k)^4}.$\end{proof}

By setting $A^{\ast}_i = P\cap \mathcal{A}_{9k + i}$ and $B^{\ast}_i = P\cap \mathcal{B}_{9k + i}$, for $1\leq i \leq k$,  we have $2k$ subsets $A^{\ast}_{1},\ldots, A^{\ast}_{k},B^{\ast}_{1},\ldots, B^{\ast}_{k}$, each of size at least $\frac{n}{(1320k)^4}$, such that every pair of subsets $\{a_1,\ldots, a_k\}$ and $\{b_1,\ldots, b_k\}$, where $a_i \in A^{\ast}_i$ and $b_i \in B^{\ast}_i$, is mutually avoiding.\end{proof}

\section{Mutually avoiding sets in higher dimensions}\label{mutrd}

In this section we will prove Theorem \ref{rd}.  Let $ P = (p_1, \ldots, p_n)$ be an $n$-element point sequence in $\mathbb{R}^d$ in general position.  The \emph{order type} of $P$ is the mapping $\chi:{P\choose d  + 1}\rightarrow \{+1,-1\}$ (positive orientation, negative orientation), assigning each $(d+1)$-tuple of $P$ its orientation.  More precisely, by setting $p_i = (a_{i,1}, a_{i,2}, \ldots, a_{i,d}) \in \mathbb{R}^d$,

\[
\chi (\{p_{i_1}, p_{i_2}, \ldots, p_{i_{d+1}} \}) = sgn \ det
\begin{bmatrix}
    1 & 1 & \dots  & 1 \\
    a_{i_1,1} & a_{i_2,1}  & \dots  & a_{i_{d+1},1} \\
    \vdots & \vdots  & \ddots & \vdots \\
    a_{i_1,d} & a_{i_2,d}  & \dots  & a_{i_{d+1},d}
\end{bmatrix},
\]
where $i_1 < i_2 < \cdots < i_{d+1}.$

\noindent Hence two point sequences $P = (p_1,\ldots, p_n)$ and $Q = (q_1,\ldots, q_n)$ have the same order-type if and only if they are ``combinatorially equivalent." See \cite{GP} and \cite{6}  for more background on order-types.

Given $k$ disjoint subsets $P_1,\ldots, P_k\subset P$, a \emph{transversal} of $(P_1,\ldots, P_k)$ is any $k$-element sequence $(p_1,\ldots, p_k)$ such that $p_i \in P_i$ for all $i$.  We say that the $k$-tuple $(P_1,\ldots, P_k)$ has \emph{same-type transversals} if all of its transversals have the same order-type.  In 1998, B\'ar\'any and Valtr proved the following same-type lemma.

\begin{Lem}[\cite{BV}]\label{same}
Let $P = (p_1,\ldots, p_n)$ be an $n$-element point sequence in $\RR^d$ in general position.  Then for $k >0$, there is an $\varepsilon = \varepsilon(d,k)$, such that one can find disjoint subsets $P_1,\ldots, P_k \subset P$ such that $(P_1,\ldots, P_k)$ has same-type transversals and $|P_i| \geq \varepsilon n$.

\end{Lem}

\noindent Their proof shows that $\varepsilon = 2^{-O(k^{d-1})}$.  This was later improved by Fox, Pach, and Suk \cite{FPS} who showed that Lemma \ref{same} holds with $\varepsilon = 2^{-O(d^3k\log k)}$.  We will use the following result, which was communicated to us by Jacob Fox, which shows that Lemma \ref{same} holds with $\varepsilon$ decaying only polynomially in $k$ for fixed $d\geq 3$.

\begin{Lem}\label{newsame}

Lemma \ref{same} holds for $\varepsilon = k^{-c_d}$, where $c_d$ depends only on $d$.
\end{Lem}

The proof of Lemma \ref{newsame} is a simple application of the following regularity lemma due to Fox, Pach, and Suk.  A partition on a finite set $P$ is called \emph{equitable} if any two parts differ in size by at most one.

\begin{Lem}[Theorem 1.3 in  \cite{FPS}]\label{reg}

For $d > 0$, there is a constant $c = c(d)$ such that the following holds.  For any $\varepsilon > 0$ and for any $n$-element point sequence $P  = (p_1,\ldots, p_n)$ in $\RR^d$, there is an equitable partition $P = P_1\cup \cdots \cup P_K$, with $1/\varepsilon < K < (1/\varepsilon)^{c}$, such that all but at most $\varepsilon {K\choose d+1}$ $(d+ 1)$-tuples of parts $(P_{i_1},\ldots, P_{i_{d+1}})$ have same-type transversals.
\end{Lem}
\noindent Let us note that $K > 1/\varepsilon$ follows by first arbitrarily partitioning $P$ into $\lceil 1/\varepsilon \rceil$ parts, such that any two parts differ in size by at most one, and then following the proof of Theorem 1.3 in \cite{FPS}.

\medskip

The next lemma we will use is Tur\'an's Theorem for hypergraphs. Given an $r$-uniform hypergraph $\mathcal{H}$, let $ex(n,\mathcal{H})$ denote the maximum number of edges in any $\mathcal{H}$-free $r$-uniform hypergraph on $n$ vertices.

\begin{Lem}[de Caen \cite{DDC}]\label{TT} Let $K^r_k$ denote the complete $r$-uniform hypergraph on $k$ vertices. Then
\begin{align*}
 \textrm{ex}\hspace{0.6mm}(n, K^r_k) \leq \left( 1- \frac{1}{\binom{k-1}{r-1} }+o(1) \right) \binom{n}{r}.
\end{align*}

\end{Lem}

\begin{proof}[Proof of Lemma \ref{newsame}]

Let $P = (p_1,\ldots, p_n)$ be an $n$-element point sequence in $\RR^d$ in general position.  Set $\varepsilon = 1/(2k)^d$, and apply Lemma \ref{reg} to $P$ with parameter $\varepsilon$ to obtain the equitable partition $P = P_1\cup \cdots \cup P_K$ with the desired properties.  Hence $|P_i| \geq n/(2k)^{d\cdot c}$, where $c$ is defined in Lemma \ref{reg}.  Since all but at most $\varepsilon {K\choose d+1}$ $(d+ 1)$-tuples of parts $(P_{i_1},\ldots, P_{i_{d+1}})$ have same-type transversals, we can apply Lemma \ref{TT} to obtain $k$ parts $P'_1,\ldots, P'_k \in \{P_1,\ldots, P_K\}$ such that all $(d+ 1)$-tuples $(P'_{i_1},\ldots, P'_{i_{d+1}})$ in $\{P'_1,\ldots, P'_k\}$ have same-type transversals.\end{proof}

\begin{proof}[Proof of Theorem \ref{rd}]
Let $k > 0$ and let $P$ be an $n$-element point set in $\RR^d$ in general position.  We will order the elements of $P  = \{p_1,\ldots, p_n\}$ by increasing first coordinate, breaking ties arbitrarily.  Let $c' = c'(d)$ be a sufficiently large constant that will be determined later.  We apply Lemma \ref{newsame} to $P$ with parameter $k' = \lceil k^{c'}\rceil$ to obtain subsets $P_1,\ldots, P_{k'}\subset P$ such that $|P_i| \geq k^{-c_dc'}n$, where $c_d$ is defined in Lemma \ref{newsame}, such that all $(d+ 1)$-tuples $(P_{i_1},\ldots, P_{i_{d+1}})$ have same-type transversals.  Let $P'$ be a $k'$-element subset obtained by selecting one point from each subset $P_i$.  By applying Theorem \ref{avoidrd} to $P'$, we obtain subsets $A,B \subset P'$ such that $A$ and $B$ are mutually avoiding, and $|A|,|B| \geq \Omega((k')^{1/(d^2 - d + 1)})$. By choosing $c' = c'(d)$ sufficiently large, we have $|A|,|B| \geq k$.  Let $\{a_1,\ldots, a_k\} \subset A$ and $\{b_1,\ldots, b_k\} \subset B$.  Then the subsets $A_1,\ldots, A_k,B_1,\ldots, B_k \in \{P_1,\ldots, P_{k'}\}$, where $a_i \in A_i$ and $b_i \in B_i$, are as required in the theorem.\end{proof}

\subsection{Crossing Families in Higher Dimensions}

Let $P$ be an $n$-element point set in $\RR^d$ in general position.  A \emph{$(d-1)$-simplex} in $P$ is a $(d-1)$-dimensional simplex generated by taking the convex hull of $d$ points in $P$. We say that two $(d-1)$-simplices \emph{strongly cross} in $P$ if their interiors intersect and they do not share a common vertex.  A \emph{crossing family} of size $k$ in $P$ is a set of $k$ pairwise strongly crossing $(d-1)$-simplices in $P$.

In \cite{1}, Aronov et al.~stated that Theorem \ref{avoidrd} implies that every point set $P$ in $\RR^d$ in general position contains a polynomial-sized crossing family, that is, a collection of $(d-1)$-simplices in $P$ such that any two strongly cross.  Since they omitted the details, below we provide the construction of a crossing family using  mutually avoiding sets in $\mathbb{R}^d.$

\begin{Cor}\label{rdcross}
Let $d\geq 2$ and let $P$ be a set of $n$ points in $\RR^d$ in general position.  Then $P$ contains a crossing family of size $\Omega(\sqrt{n})$ for $d=2$, and of size $\Omega_d(n^{\frac{1}{2\prod_{i=3}^{d}(i^2-i+1)}})$ for $d \geq 3.$
\end{Cor}

\begin{proof}
We proceed by induction on $d$. The base case $d=2$ follows from Theorem \ref{crossingr2}: a pair of mutually avoiding sets $A$ and $B$ in the plane, each of size $\Omega(\sqrt{n})$, gives rise to a crossing family of size $\Omega({\sqrt{n}}).$  For the inductive step, assume the statement holds for all $d' < d$.

Let $P$ be a set of $n$ points in $\RR^d$ in general position.  By Theorem \ref{avoidrd}, there is a pair of mutually avoiding sets $A$ and $B$ such that $|A| = |B| = k=\Omega_d({n^{\frac{1}{d^2-d+1}}})$.  Let $A = \{a_1,\ldots, a_k\}$ and $B = \{b_1,\ldots, b_k\}$.  Since $conv(A)\cap conv(B) = \emptyset$, by the separation theorem (see Theorem~1.2.4 in \cite{6}), there is a hyperplane $\mathcal{H}$ such that $A$ lies in one of the closed half-spaces determined by $\mathcal{H}$, and $B$ lies in the opposite closed half-space.

For each $a_i \in A$, let $a_ib$ be the line generated by points $a_i$ and $b \in B$.  Then set $B_i  = \{a_ib\cap \mathcal{H}: b \in B\}$.  Since $P$ is in general position, $B_i$ is also in general position in $\mathcal{H}$ for each $i$.  Moreover, since $A$ and $B$ are mutually avoiding, $B_i$ has the same order-type as $B_j$ for every $i \neq j$.  Indeed, for any $d$-tuple $b_{i_1},b_{i_2},\ldots, b_{i_d} \in B$, every point in $A$ lies on the same side of the hyperplane generated by $b_{i_1},b_{i_2},\ldots, b_{i_d}$.  Hence the orientation of the corresponding $d$-tuple in $B_i\subset \mathcal{H}$ will be the same as the orientation of the corresponding $d$-tuple in $B_j\subset \mathcal{H}$ for $i\neq j$.  Therefore, let us just consider $B_1\subset \mathcal{H}$. By the induction hypothesis, there exists a crossing family of $(d-2)$-simplices of size $$k'=\Omega_d\left(k^{\frac{1}{2\prod_{i=3}^{d-1}(i^2-i+1)}}\right)=\Omega_d\left(n^{\frac{1}{2\prod_{i=3}^{d}(i^2-i+1)}}\right),$$ in $B_1 \subset \mathcal{H}.$ Let $\mathcal{S}=\{\mathcal{S}_1,\ldots, \mathcal{S}_{k'}\}$ be our set of pairwise crossing $(d-2)$-simplices in $B_1\subset \mathcal{H}$ and let $\mathcal{S}' = \{\mathcal{S}'_1,\ldots, \mathcal{S}'_{k'}\}$ be the corresponding $(d-2)$-simplices in $B$ (which may or may not intersect).

Set $\mathcal{S}_i^{\ast} = \conv(a_i\cup \mathcal{S}'_i)$.  Then $\mathcal{S}^{\ast}_1,\ldots, \mathcal{S}^{\ast}_{k'}$ is a set of $k'$ pairwise crossing $(d -1)$-simplices in $\RR^d$.  Indeed, consider $\mathcal{S}^{\ast}_i$ and $\mathcal{S}^{\ast}_j$.  If $\mathcal{S}'_i\cap \mathcal{S}'_j \neq \emptyset$, then we are done.  Otherwise, we would have $\mathcal{S}'_j \cap \mathcal{S}^{\ast}_i \neq \emptyset$ or $\mathcal{S}'_i \cap \mathcal{S}^{\ast}_j \neq \emptyset$ since $B_i$ and $B_j$ have the same order type and $\mathcal{S}_i\cap \mathcal{S}_j \neq \emptyset$. More precisely, let $r_i$ be a ray from $a_i$ through an intersection point of $\mathcal{S}_i$ and $\mathcal{S}_j.$ The ray $r_i$ intersects both $\mathcal{S}^{'}_{i}$ and $\mathcal{S}^{'}_{j}$ by the definition of $\mathcal{S}_i$ and $\mathcal{S}_j.$ Without loss of generality assume $r_i$ intersects $\mathcal{S}_{i}$ first. It follows that $\mathcal{S}^{'}_{i} \cap \mathcal{S}^{\ast}_{j} \neq \emptyset.$ \end{proof}

\bigskip
\noindent \textbf{Acknowledgment.} We would like to thank the referees for their helpful remarks.

\bibliographystyle{amsplain}

\end{document}